\documentclass[12pt]{article}
\usepackage[table]{xcolor}
\usepackage{soul}
\usepackage{amsfonts}
\usepackage{amsthm}
\usepackage{amssymb}
\usepackage{graphicx, enumerate}
\usepackage{amsmath}
\usepackage{latexsym}
\usepackage{longtable}
\usepackage{tabularx}
\usepackage{amsmath}
\usepackage{amsfonts}
\usepackage{amsthm}
\usepackage{setspace}
\usepackage{graphicx}
\usepackage{float}
\usepackage{rotating}
\usepackage{tikz}
\usepackage{verbatim}

\newtheorem{thm}{Theorem}[section]
\newtheorem{lem}[thm]{Lemma}
\newtheorem{obs}[thm]{Observation}
\newtheorem{prop}[thm]{Proposition}

\newtheorem{prb}[thm]{Open Problem}

\makeatletter
\def\imod#1{\allowbreak\mkern10mu({\operator@font mod}\,\,#1)}
\makeatother

\newcommand{\MRS}{\mathrm{MRS}}
\newcommand{\LMRS}{\mathrm{LMRS}}
\newcommand{\SMS}{\mathrm{SMS}}
\newcommand{\MS}{\mathrm{MS}}
\newcommand{\MR}{\mathrm{MR}}

\begin{document}
\title{Note on a magic rectangle  set on dihedral group }

\author{Sylwia Cichacz \\AGH University of Krak\'ow, Poland}

\maketitle
\begin{abstract}
Let $\Gamma$ be a group of order $mnk$ and $\MRS_{\Gamma}(m,n;k)=(a_{i,j}^s)_{m\times n}$ be a collection of $k$ arrays $m\times n$  whose entries are all distinct elements of $\Gamma$. If there exist elements $\rho,\sigma\in\Gamma$ such that for every row $i$, there exists an ordering of elements such that 
	$$
		a_{i,j_1}^s a_{i,j_2}^s \dots a_{i,j_{n-1}}^s a_{i,j_n}^s= \rho
	$$
	and for every column $j$ there exists an ordering of elements such that 
	$$
		a_{i_1,j}^s a_{i_2,j}^s \dots a_{i_{m-1},j}^s a_{i_m,j}^s = \sigma,
	$$
	then $\MRS_{\Gamma}(m,n;k)$ is called a \emph{$\Gamma$-magic rectangle  set}.
    
We investigate magic rectangle sets over dihedral groups and prove that 
$\mathrm{MRS}_{\Gamma}(m,n;k)$ exists for every dihedral group $\Gamma$ of order $mnk$, 
provided that $m$ and $n$ are even. As a consequence, we obtain broad existence 
results for magic rectangles and magic squares over dihedral groups. 

\end{abstract}

 \section{Introduction}
A \emph{magic square} of order $n$ is an $n \times n$ array containing the integers $1,2,\ldots,n^2$, each appearing exactly once, such that the sum of every row, every column, and both main diagonals is equal to the constant $n(n^2+1)/2$.  
The earliest known example is the $3 \times 3$ \emph{Lo Shu magic square}, recorded in ancient Chinese literature as early as 2800 B.C. Since then, magic squares and their numerous generalizations have been extensively studied. 

In recent years, algebraic generalizations of magic-type structures have been investigated, where the additive condition is replaced by multiplication in a group (see~\cite{CicJaco,CicFro,CicFro2,Cichacz-Hinc-2,Evans,Froncek_dih_SMS_n=0mod4,Pellegrini,Sun-Yihui,ref_YuFengLiu}). In this paper, we focus on magic rectangles defined over groups.

Let $\Gamma$ be a group of order $mnk$ and a $\Gamma$-\textit{magic rectangle set} $\MRS_{\Gamma}(m,n;k)=(a_{i,j}^s)_{m\times n}$ be a collection of $k$ arrays $m\times n$  whose entries are all distinct elements of $\Gamma$. If there exist elements $\rho,\sigma\in\Gamma$ such that for every row $i$, there exists an ordering of elements such that 
	$$
		a_{i,j_1}^s a_{i,j_2}^s \dots a_{i,j_{n-1}}^s a_{i,j_n}^s = \rho
	$$
	and for every column $j$ there exists an ordering of elements such that 
	$$
		a_{i_1,j}^s a_{i_2,j}^s \dots a_{i_{m-1},j}^s a_{i_m,j}^s = \sigma,
	$$
	then $\MRS_{\Gamma}(m,n;k)$ is called a \emph{$\Gamma$-magic rectangle  set}.
	If, for every row and column, the ordering is
	\begin{equation*}
		a_{i,1}^s a_{i,2}^s \dots a_{i,{n-1}}^s a_{i,n}^s = \rho
	\end{equation*}
	and  
	\begin{equation*}
		a_{m,j}^s a_{m-1,j}^s \dots a_{{2},j}^s a_{1,j}^s = \sigma,
	\end{equation*}
	then $\MRS_{\Gamma}(m,n;k)$ is \emph{linearly $\Gamma$-magic denoted by $\LMRS_{\Gamma}(m,n;k)$.} 

    Let $\Gamma$ be a group of order $n^2$ and $\MR_{\Gamma}(n,n)=MRS_{\Gamma}(n,n;1)$ be a {$\Gamma$-magic rectangle.} If the row and column products are equal, that is, $\rho=\sigma$, then we call $\MR_{\Gamma}(n,n)$ a \emph{$\Gamma$-semi-magic square} and if moreover the products of both the main and backward main diagonals are equal to $\rho=\sigma$, then $\MR_{\Gamma}(n,n)$ is called a \emph{$\Gamma$-magic square}.

The dihedral group $D_{l}$ of order $2l$ consists of $l$ \emph{rotations} $r_i$ and $l$ \emph{reflections} $s_i$. The rotations form a cyclic subgroup of order $l$, and each reflection generates a subgroup of order 2.  

Formally, for $l \ge 3$, the \emph{dihedral group} $D_l$ can be defined as the set
\[
\{ r_0, r_1, \dots, r_{l-1}, s_0, s_1, \dots, s_{l-1} \},
\]  
where $r_0 = e$ is the identity, $r_1 = r$, $r_i = r^i$ for $i = 0,1,\dots,l-1$, $s_0 = s$, $s_i = r^i s$, $s_i^2 = e$, and the multiplication satisfies
\[
r^i s = s r^{-i}, \quad \text{for } i=0,1,\dots,l-1.
\]

An important property of $D_l$ will be used in our constructions. It follows directly from the definition.

\begin{prop}\label{prop:sr^is}
	In any dihedral group $D_l$, we have
	\[
	s r^i s = r^{-i}, \quad \text{for every } i=0,1,\dots,l-1.
	\]
\end{prop}
In \cite{Cichacz-Hinc-2} authors asked about  necessary and sufficient conditions for existence of $\LMRS_{\Gamma}(m,n;k)$, where $\Gamma$ is a non-Abelian group of order $mnk$. Recently, semi-magic squares over dihedral groups were studied in \cite{CicFro2}, where the authors proved the following result.
\begin{thm}[\cite{CicFro2}]\label{thm:mainCicFro}
	There exists a $\Gamma$-semi-magic square $\SMS_{\Gamma}(n)$, where $\Gamma$ is a dihedral group, {if and only if $n$ is even and} $n\geq4$.
\end{thm}

They also posed several open problems concerning linear constructions, fully magic squares, and general magic rectangles over dihedral groups.
\begin{prb}[\cite{CicFro2}]\label{prb:linear}
	Construct linearly $\Gamma$-semi-magic squares $\SMS_{\Gamma}(n)$, where $\Gamma$ is a dihedral group, for every  $n\geq4$.
\end{prb}
\begin{prb}[\cite{CicFro2}]\label{prb:all-gamma-semimagic}
	Construct $\Gamma$-magic squares $\MS_{\Gamma}(n)$, where $\Gamma$ is a dihedral group, for every even $n$, $n\geq4$.
\end{prb}
\begin{prb}[\cite{CicFro2}]\label{prb:all-gamma-magic}
	Characterize  $\Gamma$-magic rectangles $\MS_{\Gamma}(m,n)$, where $\Gamma$ is a dihedral group of order $mn/2$.
\end{prb}
The aim of this paper is to address the above problems. 
We construct linearly $\Gamma$-semi-magic squares $\SMS_\Gamma(n)$ as well as 
$\Gamma$-magic squares $\MS_\Gamma(n)$, where $\Gamma$ is a dihedral group, 
for every integer $n \ge 4$ with $n \equiv 0 \pmod{4}$. 
Moreover, we prove that for all even integers $m$ and $n$, 
a magic rectangle $\MR_{\Gamma}(m,n)$ exists when $\Gamma$ is the dihedral group $D_l$. 
Conversely, if $l$ is odd, then no such magic rectangle $\MR_{D_l}(m,n)$ exists.

\section{A $D_l$-Magic Rectangle Set}

We begin with several simple observations.
\begin{obs}\label{obs1}
 If a $\Gamma$-magic rectangle set $\MRS_{\Gamma}(m, n; k)$ on group $\Gamma$ exists, then both $\MRS_{\Gamma}(m, nk;1)$ and
$\MRS_{\Gamma}(mk, n;1)$ exist.
\end{obs}
\begin{proof} Let $\{M^0, M^1, \dots, M^{k-1}\}$ be a $\Gamma$-magic rectangle set $\MRS_{\Gamma}(m,n;k)$. That is, each $M^p=(a_{i,j}^p)_{m\times n}$  with entries from $\Gamma$, and for each rectangle $M^p$ for every row $i$ there exists an ordering of elements such that $$ a_{i,j_{p(1)}}^p a_{i,j_{p(2)}}^p \dots a_{i,j_{p(n)}}^p = \rho $$ and for every column $j$ there exists an ordering of elements such that $$ a_{i_{p(1)},j}^p a_{i_{p(2)},j}^p \dots a_{i_{p(m)},j}^p  = \sigma. $$
Define a single $m \times nk$ rectangle $R$ by concatenating all rectangles $M^p$ horizontally:
\[
R = [\, M^0 \mid M^1 \mid \dots \mid M^{k-1} \,],
\]  
where $M^p$ occupies columns $pn+1, \dots, (p+1)n$ of $R$.

Consider row $i$ of $R$. It is the concatenation of the $i$th rows of $M^0,M^1 \ldots, M^{k-1}$. Using the orderings from each $M^s$, the product along row $i$ is  
    \[
    \prod_{p=0}^{k-1} \Big( a_{i,j_{p(1)}}^p a_{i,j_{p(2)}}^p \dots a_{i,j_{p(n)}}^p \Big) = \underbrace{\rho \rho \cdots \rho}_{k\text{ times}} = \rho^k,
    \]  
    which is a fixed group element independent of the row.
Each column of $R$ comes from a column of some $M^p$. Since the column products in $M^p$ satisfy the original condition, the column product in $R$ is still $\sigma$.  

Analogously we can prove that  $\MRS_{\Gamma}(mk, n;1)$ exists.\end{proof}
\begin{obs}\label{odd} Let $l$ be an odd integer and $D_l$ be a dihedral group of order $mnk$. There does not exist an $\MRS_{D_l}(m, n; k)$.
\end{obs}
\begin{proof}
Without loss of generality, we may assume that $m$ is odd. Since $|D_l| = 2l = mnk$, it follows that $nk$ is even. 

Suppose that there exists an $\MRS_{D_l}(m,n;k)$. Then, by Observation~\ref{obs1}, there exists an $\MRS_{D_l}(m, nk;1)$, which we denote by $(a_{i,j})_{m \times nk}$.  

Since every reflection in $D_l$ has the form $r^i s$ and the product of two reflections is a rotation, it follows that (after rewriting the word by repeatedly grouping reflections into pairs) the product of an odd number of reflections is always a reflection, independent of the initial ordering. Thus by Proposition~\ref{prop:sr^is}, it is clear that the product of all elements of $D_l$ cannot be a pure rotation, i.e.,
\[
\prod_{g \in D_l} g \neq r^w \quad \text{for any ordering of the elements}.
\]  

On the other hand, for each column $j$ of $(a_{i,j})_{m \times nk}$, there exists an ordering of its elements such that $a_{i_1,j} a_{i_2,j} \cdots a_{i_m,j} = \sigma.$ Hence, considering all entries of the rectangle, there exists an ordering such that
\[
\prod_{g \in D_l} g = \sigma^{kn}.
\]  
But since $kn$ is even, we have $(\sigma)^{kn} = r^h$ for some $h \in \{0,1,\dots,l-1\}$, a contradiction.
\end{proof}

Observation~\ref{odd} implies the following observations immediately:
\begin{obs}\label{aodd} If  exactly one of the numbers $m,n,k$ is even and moreover congruent to $2$ modulo $4$, then  there does not exist a magic rectangle set $\MRS_{D_l}(m, n; k)$ on any dihedral group $D_l$ of order $mnk$.
\end{obs}

\begin{obs}\label{zodd} If $l$ is odd there does not exist a magic rectangle  $\MR_{D_l}(m,n)$ for any $mn=2l$.
\end{obs}

\begin{obs}\label{zodd2} If $l$ is odd there does not exist a magic rectangle  $\MRS_{D_{2l}}(2,l;2)$.
\end{obs}
\begin{proof}Suppose, for a contradiction, that a magic rectangle
$\MRS_{D_{2l}}(2,l;2)$ exists.
Let $M^{1}$ and $M^{2}$ denote the two $2\times l$ arrays.
Let $\rho \in D_{2l}$ and $\sigma \in D_{2l}$ be the common row product
and the common column product (respectively), which are the same for
both arrays by definition.

For $p\in\{1,2\}$ and $i\in\{1,2\}$, let $m_i^p$ be the number of
reflections appearing in row $i$ of $M^p$.
Similarly, for $p\in\{1,2\}$ and $j\in\{1,\dots,l\}$, let $n_j^p$ denote
the number of reflections in column $j$ of $M^p$.

Since the product of the elements in each row is $\rho$, the parity of
the number of reflections is the same in every row of any array.
Hence $m_i^p \equiv m_{i'}^{p'} \pmod 2$. Likewise, since each column product is $\sigma$, we have
$n_j^p \equiv n_{j'}^{p'} \pmod 2$.

First, suppose that $n_j^p$ is odd.
As each array has exactly two rows, this forces
$n_j^p=1$ for each $j$ and $p$. Since $l$ is odd, this implies that the parities of
$m_1^p$ and $m_2^p$ must differ, contradicting the fact that all row
parities of reflections in $M^p$ are equal.

Now suppose that $n_j^p$ is even.
As before, since there are only two rows, we must have
$n_n^p\in\{0,2\}$, which implies $m_1^p=m_2^p$ for $p\in\{1,2\}$. Since all row parities of the number of reflections are equal, we obtain $m_1^1+m_2^1+m_1^2+m_2^2 \equiv 0 \pmod 4$. On the other hand, the total number of entries in both arrays is $2l$,
which is congruent to $2 \pmod 4$ because $l$ is odd.
This contradiction completes the proof.
\end{proof}

Now we show that for $m$ and $n$ both even, a linearly magic rectangle set $\MRS_{D_l}(m, n; k)$  can be
constructed for any $k$ and any dihedral group $D_l$ of order $mnk$. Observe that, since $m$ and $n$ are even, it follows that  $l$ is even as well, because $2l=mnk$.
We start with an elementary step.
\begin{lem}\label{dwa} Let $l>1$ be an  integer and $D_{2l}$ be a dihedral group of order $4l$. There  exists an $\LMRS_{D_{2l}}(2, 2; l)$ such that the row product is $\rho=rs,$ the column product is $\sigma=s$.
\end{lem}
\begin{proof} Let  $M^p=(m^p_{i,j})_{2\times 2}$ for $p\in\{0,1,\ldots,l-1\}$. Set  \\

 \begin{figure}[H]
 \begin{center}
 $M^p$=\begin{tabular}{|c|c|}
 \hline

  $r^{2p+1}$  & $r^{-2p}s$ \\ \hline
$r^{2p+1}s$  & $r^{2p}$ \\ \hline
   \end{tabular}
 \caption{Square $M^p$\label{fig:Mp}}
 \end{center}
 \end{figure}

Observe that $$\rho^p_i=a_{i,1}^pa_{i,2}^p=rs,$$ $$\sigma^p_j=a_{2,j}^pa_{1,j}^p=s,$$

This finishes the proof.\end{proof}

\begin{thm}
If $m,n$ are both even, then a linearly magic rectangle set $\MRS_{D_{2l}}(m, n; k)$ exists
for every $k$ and $4l=mnk$. Moreover, the row product is $\rho=(rs)^{n/2}$ and the column product is $\sigma=s^{m/2}$. \label{glowne}
\end{thm}
\begin{proof} There exists an  $\MRS_{\Gamma}(2, 2;l)$ by Lemma~\ref{dwa}.
Let $\{M^0, M^1, \dots, M^{l-1}\}$ be a $\Gamma$-magic rectangle set $\MRS_{\Gamma}(2,2;l)$ which exists by Lemma~\ref{dwa}.

Let $m'=m/2$ and $n'=n/2$. For $u\in\{0,1,\ldots, k-1\}$ define a single rectangle $R^u$ of size $m \times n$ by concatenating the rectangles $M^{um'n'},M^{um'n'+1} \dots,$ $ M^{(u+1)m'n'-1}$ in the way as is shown in Figure~\ref{fig:Ru}.

 \begin{figure}[H]
 \begin{center}
 $R^u$=\begin{tabular}{|c|c|c|c|}
 \hline
     $M^{um'n'}$  & $M^{um'n'+1}$&$\dots$&$M^{um'n'+n'-1}$ \\ \hline
     $M^{um'n'+n'}$  & $M^{um'n'+n'+1}$&$\dots$&$M^{um'n'+2n'-1}$ \\  \hline
     $\vdots$& $\vdots$&$\dots$&$\vdots$\\ \hline
 $M^{(u+1)m'n'-n'}$  & $M^{(u+1)m'n'-n'+1}$&$\dots$&$M^{(u+1)m'n'-1}$ \\  \hline

   \end{tabular}
 \caption{A square $R^u$}\label{fig:Ru}
 \end{center}
 \end{figure}
That is, each $R^u=(b_{i,j}^u)_{m\times n}$  with entries from $D_{2l}$, and for each rectangle $R^u$ for every row $i$ there is $$b_{i,1}^u b_{i,2}^u \dots b_{i,n-1}^ub_{i,n}^u = (rs)^{n'}=(rs)^{n/2} $$ 
and for every column $j$ there is $$b_{m,j}^u b_{m-1,j}^u \dots b_{2,j}^ub_{1,j}^u = s^{m'}=s^{m/2}. $$ 

\end{proof}

As an immediate consequence of Theorem~\ref{glowne}, we have.
\begin{thm}\label{linear}
    There exists an $\mathrm{LSMS}_\Gamma(n)$, where $\Gamma$ is a dihedral group, for every $n\equiv 0\pmod4$, $n\geq4$. Moreover, for $n\equiv 0\pmod 8$ it is a {linearly $\Gamma$-magic square}.
\end{thm}
\begin{proof}
By Theorem~\ref{glowne}, there exists an $\mathrm{LMRS}_{\Gamma}(n,n;1)$ for the dihedral group $\Gamma$ of order $n^2$. Moreover, the row product is $\rho = (rs)^{n/2}$ and the column product is $\sigma = s^{n/2}$. Since $n \equiv 0 \pmod 4$, we have $n/2$ even. In the dihedral group, every reflection has order $2$, therefore, we have $\rho = \sigma = r^0$.

Assume now that $n\equiv 0\pmod 8$. Thus $n=8k$ and $\Gamma=D_{32k^2}$ for $k\geq 1$. Observe that in the proof of Lemma~\ref{dwa} every square $M^p$ of size $2\times 2$ has the diagonal sums $$\delta^p_1=a_{1,1}^pa_{2,2}^p=r^{2p+1}r^{2p}=r^{4p+1}, \;\;\;\;\delta^p_2=a_{1,2}^pa_{2,1}^p=r^{2p+1}sr^{-2p}s=r^{4p+1}$$
Let $A=\{0,2,6,7\}$ for $k=1$ and \[
A=\{1,8k^2-1,2,8k^2-2,\dots,(2k-1),8k^2-(2k-1),0,8k^2-k\}
\] 
otherwise. Observe that $|A|=4k$, $A\subset \{0,1,\ldots,8k^2-1\} $
and $\sum\limits_{a\in A}a\equiv-k\pmod{8k^2}$.
Let \[
B=\{a+8k^2:\ a\in A\}.
\] 
This implies $|B|=4k$,  $B\subset \{8k^2,8k^2+1\ldots,16k^2-1\} $, $\sum\limits_{b\in B}b\equiv-k\pmod{8k^2}$. Therefore  $$\prod_{a\in A}\delta^a_1=\prod_{a\in A}r^{4a+1}=r^{\sum_{a\in A}(4a+1)}=r^{-4k\pmod{32k^2}+4k}=r^0$$
and analogously $\prod_{b\in B}\delta^b_2=r^0.$

 If now we define a  $8k \times 8k$ square by concatenating all squares $M^p$ for $p\in\{0,1,\ldots,16k^2-1\}$ such that on the main diagonal are squares $M^a$ for $a\in A$ and on the backward diagonal $M^b$ for $b\in B$, then we obtain an $\mathrm{LMS}_{D_{32k^2}}(8k)$, which completes the proof.
\end{proof}

 \begin{figure}[H]
$$
\begin{array}{||c|c||c|c||c|c||c|c||}
\hline\hline
\cellcolor{lightgray}{r} & \cellcolor{lightgray}{r^0s}  & r^3  & r^{-2}s & r^7 & r^{-6}s  &\cellcolor{yellow} {r^{17}}  & \cellcolor{yellow}{r^{-16}s}\\ \hline
\cellcolor{lightgray}{rs}  & \cellcolor{lightgray}{r^{0}} & r^{3}s & r^{2}& r^{7}s  & r^{6} & \cellcolor{yellow}{r^{17}s} & \cellcolor{yellow}{r^{16}s} \\ \hline\hline
r^{9}  &r^{-8}s  &\cellcolor{lightgray}{r^{ 5}}  & \cellcolor{lightgray}{r^{-4}s}&\cellcolor{yellow}{r^{31}}  &\cellcolor{yellow}{r^{-30}s}  &r^{ 11}  & r^{-10}s \\ \hline
r^{9}s  & r^{8}s & \cellcolor{lightgray}{r^{5}s} & \cellcolor{lightgray}{r^{4}} &\cellcolor{yellow}{r^{31}s}  & \cellcolor{yellow}{r^{30}} & r^{11}s & r^{10} \\ \hline\hline
r^{19} & r^{-18}s & \cellcolor{yellow}{r^{21}} & \cellcolor{yellow}{r^{-20}s}&\cellcolor{lightgray}{r^{15}}  & \cellcolor{lightgray}{r^{-14}s} & r^{23} & r^{-22}s \\ \hline
r^{ 19}s  & r^{18}  &\cellcolor{yellow}{r^{21}s}  & \cellcolor{yellow}{r^{20}}&\cellcolor{lightgray}{r^{15}s}  &\cellcolor{lightgray}{r^{ 14}}  &r^{23}s  & r^{22} \\ \hline\hline
\cellcolor{yellow}{r^{29}}  & \cellcolor{yellow}{r^{-28}s} & r^{25} & r^{-24}s &r^{27}  & r^{-26}s & \cellcolor{lightgray}{r^{13}}  & \cellcolor{lightgray}{r^{-12}s}  \\ \hline
\cellcolor{yellow}{r^{29}s}  &\cellcolor{yellow}{r^{ 28}}  &r^{25}s  & r^{24}&r^{27}s  &r^{ 26} & \cellcolor{lightgray}{r^{13}s}  &\cellcolor{lightgray}{r^{12}}\\ \hline\hline
\end{array}
$$

 \caption{ $\mathrm{LMS}_{D_{32}}(8)$ with the magic constant $\mu=r^{0}$}

 \end{figure}

\begin{thm}\label{linear}
    There exists a  $\Gamma$-magic square $\mathrm{MS}_\Gamma(n)$, where $\Gamma$ is a dihedral group of order $n^2$, for every $n\equiv 0\pmod4$, $n\geq4$.
\end{thm}
\begin{proof} Let $n=4k$ for $k\geq 1$.
Let  $\tilde{M}^p=(\tilde{m}^p_{i,j})_{2\times 2}$ for $p\in\{0,1,\ldots,2k^2-1\}$ be defined as  \\

 \begin{figure}[H]
 \begin{center}
 $\tilde{M}^p$=\begin{tabular}{|c|c|}
 \hline
 $r^{2p-1}s$  & $r^{-2p}$ \\ \hline
     $r^{2p-1}$  & $r^{2p}s$ \\ \hline

   \end{tabular},
 \end{center}
 \end{figure}.

For such a square we have  
$$\rho^p_1=a_{1,2}^pa_{1,1}^p=r^{-2p}r^{2p-1}s=r^{-1}s,\;\;\;\rho^p_2=a_{2,2}^pa_{2,1}^p=r^{2p}sr^{2p-1}=rs,$$
$$\sigma^p_1=a_{1,1}a_{2,1}=r^{2p-1}sr^{2p-1}=s,\;\;\; \sigma^p_2=a_{1,2}a_{2,2}=r^{-2p}r^{2p}s=s$$
$$\delta^p_1=a_{2,2}^pa_{1,1}^p=r^{2p}sr^{2p-1}s=r,\;\;\;\delta^p_2=a_{1,2}^pa_{2,1}^p=r^{2p-1}sr^{2p}s=r^{-1}.$$

If now   for  $p\in\{2k^2,2k^2+1,\ldots,4k^2-1\}$  we define  $\tilde{M}^p$ as
 \begin{figure}[H]
 \begin{center}
 $\tilde{M}^p$=\begin{tabular}{|c|c|}
 \hline
      $r^{2p-1}$  & $r^{2p}s$ \\ \hline
 $r^{2p-1}s$  & $r^{-2p}$ \\ \hline

   \end{tabular},
 \end{center}
 \end{figure}
then one can easily see that there exists 
$$\rho^p_1=rs,\;\;\;\rho^p_2=r^{-1}s,\;\;\;\sigma^p_1=\sigma^p_2=s,\;\;\; \delta^p_1=r^{-1},\;\;\;\delta^p_2=r.$$
 Define the square $R$ of size $2k \times 2k$ by concatenating the rectangles $\tilde{M}^0,\tilde{M}^1, \tilde{M}^2 , \dots,\tilde{M}^{4k^2-1}$ in the way as is shown in Figure~\ref{fig:sR}.

 \begin{figure}[H]
 \begin{center}
 $R$=\begin{tabular}{|c|c|c|c|}
 \hline
     $\tilde{M}^0$ & $\tilde{M}^1$&$\dots$&$\tilde{M}^{2k-1}$ \\ \hline
 $\tilde{M}^{2k}$ & $\tilde{M}^{2k+1}$&$\dots$&$\tilde{M}^{4k-1}$ \\ \hline  
 $\vdots$&$\vdots$&$\dots$&$\vdots$\\\hline
    $\tilde{M}^{2k^2-2k}$ & $\tilde{M}^{2k^2-2k+1}$&$\dots$&$\tilde{M}^{2k^2-1}$ \\ \hline\hline
 $\tilde{M}^{2k^2}$ & $\tilde{M}^{2k^2+1}$&$\dots$&$\tilde{M}^{2k^2+2k-1}$ \\ \hline 
  $\vdots$&$\vdots$&$\dots$&$\vdots$\\\hline
   $\tilde{M}^{4k^2-2k}$ & $\tilde{M}^{4k^2-2k+1}$&$\dots$&$\tilde{M}^{4k^2-1}$ \\ \hline

   \end{tabular}
 \caption{A square $R$}\label{fig:sR}
 \end{center}
 \end{figure}

 Taking now the row, column, the main diagonal and the backward diagonal products according to the products in $\tilde{M}^p$, we obtain the row product, column  and both diagonal product $\mu=r^0$, which finishes the proof.

\end{proof}

\begin{figure}[H]
$$
\begin{array}{|c|c|}
\hline
\tilde{M}^{0}&\tilde{M}^{1}\\\hline
\tilde{M}^{2}&\tilde{M}^{3}\\\hline

\hline
\end{array}
=\begin{array}{||c|c||c|c||}
\hline\hline
r^{-1}s&r^0& r^{1}s&r^{-2}\\\hline
r^{-1}&r^0s& r^{1}&r^2s\\\hline\hline
r^{3}&r^4s& r^{5}&r^6s\\\hline
r^{3}s&r^{-4}& r^{5}s&r^{-6}\\\hline\hline

\hline
\end{array}
$$

 \caption{ $\mathrm{MS}_{D_{8}}(4)$ with the magic constant $\mu=r^{0}$}

 \end{figure}
\section{Conclusions}
In this note, we study magic rectangle sets over dihedral groups. In particular, we prove a general existence result: for any dihedral group $D_l$ of order $2l=mnk$, a $D_l$-magic rectangle set $\MRS_{D_l}(m,n;k)$ exists whenever both $m$ and $n$ are even, while for odd $l$ such a construction is impossible. Several natural directions remain open. Perhaps the most compelling is to obtain a complete characterization (up to natural equivalences) of $D_l$-magic rectangle sets.
\section{Statements and Declarations}
The work of the  author was    supported by  AGH University of Krakow under grant no. 16.16.420.054,
funded by the Polish Ministry of Science and Higher Education.
\bibliographystyle{plain}

\end{document}